\documentclass[12pt]{amsart}
\usepackage{amscd}
\usepackage{amsmath}
\usepackage{amsthm}
\usepackage{amssymb}
\usepackage{mathrsfs}
\usepackage{amsfonts}
\usepackage{color}

\usepackage{pifont}

\usepackage{upgreek}
\usepackage{bm}

\numberwithin{equation}{section}
\newtheorem{theo}{Theorem} 

\newtheorem{lem}{Lemma}
\newtheorem{mcor}{Corollary}
\newtheorem{remark}{Remark}
\newtheorem{prop}{Proposition}
\newtheorem{definition}{Definition}

\newcommand*{\tr}{\mathrm{tr}}

\newcommand*{\D}[1]{\ensuremath{\nabla^{#1}}}
\newcommand*{\tD}[1]{\ensuremath{\widetilde{\nabla}^{#1}}}

\begin{document}
\title{Adiabatic limit and connections in Finsler Geometry }
\author[Huitao Feng]{Huitao Feng$^1$}
\author[Ming Li]{Ming Li$^2$}

\address{Huitao Feng: School of Mathematics and Statistics,
 Chongqing University of Technology,
Chongqing 400054, People's Republic of China }

\email{fht@cqut.edu.cn}

\address{Ming Li: School of Mathematics and Statistics,
 Chongqing University of Technology,
Chongqing 400054, People's Republic of China }

\email{mingli@cqut.edu.cn}

\thanks{$^1$~Partially supported by NSFC (Grant No. 10921061), Fok Ying Tong
Education Foundation, Chongqing NSF and the Scientific research project of
CQUT}

\thanks{$^2$~Partially supported by Chongqing NSF (Grant No. cstc2011jjA00026)
and the Scientific research project of
CQUT (Grant No. 01-60-37)}


\date{}  
\maketitle

\begin{abstract}
In this paper, we identify the Bott connection on the natural
foliation of the projective sphere bundle of a Finsler manifold to
the Chern connection of this manifold. As a consequence,
the symmetrization of the Bott connection turns out to be the Cartan
connection of the Finsler manifold. Following Liu-Zhang
\cite{LiuZ}, the Cartan connection can also be obtained through an
adiabatic limit process.
Furthermore, a Chern-Simons type form is defined and its conformal
properties are discussed.

\begin{flushleft}
Keywords:  Bott connection,\quad  Chern connection,\quad  Cartan
connection,\quad adiabatic limit,\quad Chern-Simons type form
\end{flushleft}
\end{abstract}

\section*{Introduction}

In Finsler geometry, the Chern connection and the Cartan connection
are two basic connections which have remarkable  properties. Let
$(M,F)$ be a Finsler manifold. Let $\pi:SM\rightarrow M$ be the
projective sphere bundle of $M$. Then the Finsler structure $F$ on
$M$ defines naturally an Euclidean structure on the pull-back vector
bundle $\pi^*TM\rightarrow SM$ and a Sasaki-type Riemannian metric
on $SM$. The Chern and Cartan connections are connections on
$\pi^*TM$ and defined from different geometric reasons.

On the other hand, the Finsler structure $F$ gives rise to a natural
splitting of $T(SM)$. One part is the vertical tangent bundle
$V(SM)$ formed by the tangent vectors of the (vertical) projective
spheres, which is an integrable subbundle of $T(SM)$. Another part
is the horizontal tangent bundle $H(SM)$, which is defined as the
orthogonal complement of $V(SM)$ in $T(SM)$ with respect to the
Sasaki-type Riemannian metric on $SM$. It is well-known that $H(SM)$
with its restriction metric is isometric to $\pi^*TM$.

In this paper, we consider $SM$ as a foliation foliated by
projective spheres. So the well-known Bott connection in foliation
theory is now a connection on $H(SM)$. We will prove that the Bott
connection is the Chern connection under the identification of
$H(SM)$ and $\pi^*TM$. As a consequence, the symmetrization of the
Bott connection turns out to be the Cartan connection. These also
partially answer a question of M. Abate and G. Patrizio (cf. [1,
p.29]). Following Liu-Zhang \cite{LiuZ}, the relations between the
Chern connection, the Cartan connection and the Levi-Civita
connection associated to the Sasaki-type Riemannian metric are also
established through an adiabatic limit process.

We then consider a special Chern-Simons transgressed term of the Chern and
Cartan connections.  In the case of dimension 2, the explicit
formula of this term is given. Inspired by this formula, we define a
Chern-Simons type form of $(M,F)$, which is a non-Riemannian
geometric invariant of the Finsler manifold. Some conformal
properties of this form are also discussed.

This paper is organized as follows. In Section 1, we give a review
of some basic facts in Finsler geometry.
In Section 2, we study the relations between the Bott connection,
the Chern connection and the Cartan connection for a Finsler manifold.
In Section 3, we define a Chern-Simons type form of a Finsler manifold and discuss its
conformal properties.

\

\

{\bf Acknowledgements.} The first author would like to thank
Professor Weiping Zhang for his consistent support and
encouragement. The authors thank Professors Kefeng Liu and Weiping
Zhang for their many helpful suggestions in preparing this paper.

\section{Finsler manifolds and related structures}
In this section we give a review of some basic facts in Finsler
geometry which will be used in this paper.

Let $M$ be an $n$ dimensional smooth manifold and $\pi:TM\to M$
the tangent bundle of $M$. Let $(U;x=(x^1,x^2,\ldots,x^n))$ be a
local coordinate system on an open subset $U$ of $M$. Then by the
standard procedure one gets a local coordinate system
$(x,y)=(x^1,\ldots,x^n,y^1,\ldots,y^n)$ on $\pi^{-1}(U)$. Set
$TM_0=TM\setminus0$, where $0$ denotes the zero section of $TM$.
Then $(x,y)$ with $y\neq 0$ is a local coordinate system on $TM_0$.
\begin{definition}\label{definition of Finsler manifolds}
A Finsler structure on $M$ is a smooth function
$F:TM_0\rightarrow\mathbb{R}^+$, which satisfies the following conditions:

{\upshape(i)} $F(x,\lambda y)=\lambda F(x,y)$, $\forall (x,y)\in TM_0$, and
$\lambda\in\mathbb{R}^+$;

{\upshape(ii)} The $n\times n$ Hessian matrix
\begin{align*}
(g_{ij})=\left(\frac{1}{2}\left[F^2\right]_{y^iy^j}\right)
\end{align*}
is positive-definite at every point of $TM_0$. A manifold $M$ with a
Finsler structure $F$ is called a Finsler manifold, and denoted by
$(M,F)$.
\end{definition}

In this paper, lower case Latin indices will run from 1 to $n$ and
lower case Greek indices will run from 1 to $n-1$. We also adopt the
summation convention of Einstein.

Let $(M,F)$ be an $n$-dimensional Finsler manifold. Set
\begin{align}
G^{i}=\frac{1}{4}g^{ij}\left(\left[F^2\right]_{y^{j}x^{k}}y^{k}
-\left[F^2\right]_{x^{j}}\right),\label{geodesic coeffi.}
\end{align}
\begin{align}
\frac{\delta}{\delta x^i}=\frac{\partial}{\partial
x^{i}}-\frac{\partial G^{j}}{\partial y^i}\frac{\partial}{\partial
y^j},\quad \frac{\delta}{\delta y^i}=F\frac{\partial}{\partial
y^{i}},\label{Gii}
\end{align}
where $(g^{ij})=(g_{ij})^{-1}$. Clearly, the vectors
\begin{align}
\left\{\frac{\delta}{\delta x^1},\frac{\delta}{\delta x^2},\ldots,
\frac{\delta}{\delta x^n}, \frac{\delta}{\delta
y^1},\frac{\delta}{\delta y^2},\ldots, \frac{\delta}{\delta
y^n}\right\}\label{basis of TM_0}
\end{align}
form a local tangent frame of $TM_0$. For another
local coordinate system $(U;{\tilde x})$ on $M$, a routine computation shows that
\begin{align}
{{\delta}\over{\delta x^i}}={{\partial{\tilde x}^j}\over{\partial
x^i}}{{\delta}\over{\delta{\tilde x}^j}},\quad{{\delta}\over{\delta
y^i}}={{\partial{\tilde x}^j}\over{\partial
x^i}}{{\delta}\over{\delta{\tilde y}^j}}. \label{delta x delta y}
\end{align}
Now by (\ref{delta x delta y}), one gets a well-defined linear map $J:T(TM_0)\to T(TM_0)$
\begin{align}
J\left(\frac{\delta}{\delta x^i}\right)=\frac{\delta}{\delta
y^i},\quad J\left(\frac{\delta}{\delta
y^i}\right)=-\frac{\delta}{\delta x^i},\label{JJ}
\end{align}
which is in fact an almost complex structure on $TM_0$. Let
\begin{align}
\left\{\delta x^1,\delta x^2,\ldots,\delta x^n,\delta y^1,\delta
y^2,\ldots,\delta y^n\right\}\label{JJJ}
\end{align}
be the dual frame of (\ref{basis of TM_0}). One has
\begin{align}
\delta x^i=dx^i,\quad \delta y^i={1\over
F}\left(dy^{i}+\frac{\partial G^{i}}{\partial
y^j}dx^{j}\right),\label{JiJi}
\end{align}
and
\begin{align}
J^*(\delta x^i)=-\delta y^i,\quad J^*(\delta y^i)=\delta
x^i,\label{JiJ}
\end{align}
where $J^*$ denotes the dual map of $J$.

Let $\pi:SM=TM_0/\mathbb{R}^+\to M$ denote the projective sphere bundle.
Now the fundamental tensor $g=g_{ij}dx^i\otimes dx^j$ defines an Euclidean metric on
the pull back bundle $\pi^*TM$ over $SM$. Note
that $\pi^*TM$ admits a distinguished global section
$l:SM\to\pi^{\ast}TM$, which is defined by
\begin{align}
l(x,[y])=\left(x,[y],\frac{y}{F(x,y)}\right).\label{l1}
\end{align}

For any local orthonormal frame $\left\{e_1,\ldots,e_n\right\}$ of
$(\pi^*TM,g)$ with $e_n=l$, let $\{\omega^{1},\cdots,\omega^{n}\}$
be the dual frame. Clearly, $\omega^i$'s can be also viewed
naturally as (local) one forms on $SM$ as well as on $TM_0$. Here $\omega^n$,
the so called Hilbert form, is a globally defined one form and $\omega^n=F_{y^i}\delta x^i$. Set
\begin{align}
\omega^{n+i}=J^*(\omega^i),\quad i=1,2,\ldots, n.\label{200}
\end{align}
The one forms $\omega^1,\omega^2,\ldots,\omega^{2n-1}$ and
$\omega^{2n}=-F_{y^i}\delta y^i$ give rise to a local coframe of
$TM_0$. Moreover, one verifies easily that the forms
$\omega^{n+\alpha}$, $\alpha=1,2,\ldots,n-1$, are actually the one
forms on $SM$ (cf. [6, p.269]) and the set
\begin{align}
\theta=\left\{\omega^1,\ldots,\omega^n,\omega^{n+1},\ldots,\omega^{2n-1}\right\}
                                                         \label{total cobasis of SM}
\end{align}
forms a local coframe of $SM$. By using the local coframe (\ref{total cobasis of SM}),
the tensor
\begin{align}
g^{T(SM)}=\sum_{i=1}^n\omega^i\otimes\omega^i+\sum_{\alpha=1}^{n-1}\omega^{n+\alpha}\otimes\omega^{n+\alpha}
\end{align}
gives a well-defined Riemannian metric on $SM$, which is called
the Sasaki-type Riemannian metric on $SM$.
Moreover, the fundamental tensor $g$ can be written as
\begin{align}
g=\sum_{i,j=1}^ng_{ij}dx^i\otimes dx^j=\sum_{i,j=1}^ng_{ij}{\delta x^i}\otimes{\delta x^j}
=\sum_{i=1}^n\omega^i\otimes\omega^i\quad\mbox{\rm on}\quad SM.
\label{gij}
\end{align}

As mentioned in the introduction of this paper, the vertical and
horizontal subbundles $V(SM)$ and $H(SM)$ of $T(SM)$ are orthogonal
to each other with respect to $g^{T(SM)}$. Let
$\{\mathbf{e}_1,\ldots,\mathbf{e}_n,\mathbf{e}_{n+1},\ldots,\mathbf{e}_{2n-1}\}$
denote the dual frame of $\theta$. Note that
\begin{align}
\{\mathbf{e}_1,\mathbf{e}_2,\ldots,\mathbf{e}_{n-1},\mathbf{e}_n\}\label{mathbf e}
\end{align}
is a local orthonormal frame of $H(SM)$.
\begin{remark}
$(\pi^*TM,g)$ can be identified with $H(SM)$ with the restricting
metric of $g^{T(SM)}$ as Euclidean bundles. In fact, this
identification is given by identifying $\frac{\partial}{\partial
x^i}$ with $\frac{\delta}{\delta x^i}$ and so $e_i$ with
$\mathbf{e}_i$. In particular, the distinguished section
$l=e_n={{y^i}\over F}{\partial\over{\partial x^i}}$ in (\ref{l1})
turns out to be the Reeb vector field ${\bf
G}={\mathbf{e}_n}={{y^i}\over F}{\delta\over{\delta x^i}}$ of
$(M,F)$ on $SM$.
\end{remark}
Write that
\begin{align}
\omega^j=v^{j}_i\delta x^i,\quad{\rm and\quad so}\quad\omega^{n+\alpha}=J^*(v^{\alpha}_i\delta
x^i)=-v^{\alpha}_i\delta y^i.\label{omega to dx}
\end{align}
Then one has
\begin{align}
\mathbf{e}_i=u_{i}^j\frac{\delta}{\delta x^j}\quad{\rm and}\quad\mathbf{e}_{n+\alpha}
=-u_{\alpha}^j\frac{\delta}{\delta y^j},\label{ei to delta delta x}
\end{align}
where $(u_{j}^i)=(v^{j}_i)^{-1}$. Here also note that $v^{n}_i=F_{y^i}$
and $u^i_n=\frac{y^i}{F}$. By Definition \ref{definition of Finsler manifolds},
one gets easily that
\begin{align}
\sum_{\alpha=1}^{n-1}v^{\alpha}_iv^{\alpha}_j=FF_{y^iy^j},\quad FF_{y^i}g^{ij}=y^j.
\label{500}
\end{align}

The following lemma gives an explicit expression of the exterior
derivative of the Hilbert form $\omega^n$ with respect to
the local coframe (\ref{total cobasis of SM}).
This formula is usually obtained as one of the structure equations
of the Chern connection in Finsler geometry.
\begin{lem}\label{differential of Hilbert form}
The exterior derivative of Hilbert form is given by
\begin{align}
d\omega^n=\sum_{\alpha=1}^{n-1}\omega^{\alpha}\wedge\omega^{n+\alpha}.\label{d omega^n}
\end{align}
\end{lem}
\begin{proof}
Note that
\begin{align*}
d\omega^n&=d(F_{y^i}\delta x^i)=F_{y^ix^j}\delta x^j\wedge \delta x^i+F_{y^i y^j}dy^j\wedge \delta x^i\\
&=FF_{y^i y^j}\delta y^j\wedge \delta x^i+\left(F_{y^ix^j}-\frac{\partial G^k}{\partial y^j}F_{y^iy^k}\right)\delta x^j\wedge \delta x^i\\
&=FF_{y^i y^j}\delta y^j\wedge \delta x^i+\left(\frac{\delta F}{\delta x^j}\right)_{y^i}\delta x^j\wedge \delta x^i+\frac{\partial^2G^k}
{\partial y^j\partial y^i}F_{y^k}\delta x^j\wedge \delta x^i.
\end{align*}
By (\ref{omega to dx}) and (\ref{500}), the term $FF_{y^iy^j}\delta
y^j\wedge \delta x^i
=\sum_{\alpha=1}^{n-1}\omega^{\alpha}\wedge\omega^{n+\alpha}$.
Clearly, $\frac{\partial^2G^k} {\partial y^j\partial
y^i}F_{y^k}\delta x^j\wedge \delta x^i=0$. Now the lemma follows
from the following result (cf. [3, p.36]),
\begin{align}
\frac{\delta F}{\delta x^j}=F_{x^j}-\frac{\partial G^k}{\partial y^j}F_{y^k}=0. \label{F is horizontal}
\end{align}
\end{proof}

The following lemma is actually obtained by Mo in \cite{Mo2011}. We
will give it a direct proof without using any concepts of
connections.
\begin{lem}[Mo, \cite{Mo2011}]\label{lemma 2}
The Lie derivative of the fundamental tensor $g$ along the Reeb
vector field $\mathbf{G}$ (cf. {\rm Remark 1}.) is given by
\begin{align}
\mathcal{L}_{\mathbf{G}}g
=-\sum_{\alpha=1}^{n-1}\left(\omega^{\alpha}\otimes\omega^{n+\alpha}
+\omega^{n+\alpha}\otimes\omega^{\alpha}\right).
                                                                      \label{Lie direvative of g}
\end{align}
\end{lem}
\begin{proof}
Firstly one has
\begin{align*}
\mathbf{G}(g_{ij})&=\frac{y^k}{F}\frac{\delta}{\delta
x^k}\left({\frac{1}{2}[F^2]_{y^iy^j}}\right)
=\frac{1}{2}\frac{y^k}{F}\left(\frac{\partial}{\partial
x^k}{[F^2]_{y^iy^j}}-\frac{\partial G^l}{\partial
y^k}\frac{\partial}{\partial y^l}{[F^2]_{y^iy^j}}\right)\\
&=\frac{1}{2}\frac{y^k}{F}\left(\frac{\delta[F]^2}{\delta
x^k}\right)_{y^iy^j} +{1\over F}\left(g_{lj}\frac{\partial
G^l}{\partial y^i}+g_{li}\frac{\partial
G^l}{\partial y^j}\right)\\
&=\frac{1}{F}\left(g_{lj}\frac{\partial G^l}{\partial
y^i}+g_{li}\frac{\partial G^l}{\partial y^j}\right).
\end{align*}
Then by (\ref{gij}) and Cartan homotopy formula (cf. [11, p.30]),
one has
\begin{align*}
&\mathcal{L}_{\mathbf{G}}g=\mathcal{L}_{\mathbf{G}}(g_{ij}dx^i\otimes
dx^j)\\
&=\mathbf{G}(g_{ij})dx^i\otimes
dx^j+g_{ij}\mathcal{L}_{\mathbf{G}}(dx^i)\otimes
dx^j+g_{ij}dx^i\otimes \mathcal{L}_{\mathbf{G}}(dx^j)\\
&=\frac{1}{F}\left(g_{lj}\frac{\partial G^l}{\partial
y^i}+g_{li}\frac{\partial G^l}{\partial y^j}\right)dx^i\otimes
dx^j\\
&\quad+g_{ij}\frac{dy^i}{F}\otimes
dx^j-g_{ij}\frac{y^i}{F}d\log F\otimes dx^j+g_{ij}dx^i\otimes \frac{dy^j}{F}-g_{ij}dx^i\otimes\frac{y^j}{F}d\log F \\
&=g_{ij}\delta y^i\otimes dx^j+g_{ij} dx^i\otimes\delta y^j-d\log
F\otimes F_{y^j}dx^j-F_{y^i}dx^i\otimes d\log F\\
&=-\sum_{i=1}^n\omega^i\otimes\omega^{n+i}-\sum_{i=1}^n\omega^{n+i}\otimes\omega^{i}-d\log
F\otimes \omega^n-\omega^n\otimes d\log F\\
&=-\sum_{\alpha=1}^{n-1}\left(\omega^{\alpha}\otimes\omega^{n+\alpha}+\omega^{n+\alpha}\otimes\omega^{\alpha}\right).
\end{align*}
The last equation comes from that $\omega^{2n}=-d\log F$, a direct
corollary of (\ref{F is horizontal}).

\end{proof}

\begin{remark}
We denote the Hilbert form as $\omega=\omega^n$. By Lemma
\ref{differential of Hilbert form}, one has that
$\omega\wedge(d\omega)^{n-1}\neq0$. So $\omega$ is a contact form of
$SM$.
\end{remark}

\section{The relations of some connections related to a Finsler manifold}

In this section we will use the same notations as in Section 1. Note
that there exists a natural foliation structure on the Riemannian
manifold $(SM,g^{T(SM)})$, which is foliated by the vertical bundle
$V(SM)$. Following Liu-Zhang \cite{LiuZ} and Zhang [11, Sect. 1.7],
set
\begin{align}
\mathcal{F}=V(SM),\quad\mathcal{F}^{\bot}=H(SM).\label{00}
\end{align}
Let $\nabla^{T(SM)}$ be the Levi-Civita connection on $T(SM)$
associated to the Sasaki-type Riemannian metric $g^{T(SM)}$ on $SM$.
Let $p$, $p^{\bot}$ denote the orthogonal projections from $T(SM)$ to $\mathcal {F}$, $\mathcal {F}^{\bot}$ respectively. Set
\begin{align}
\nabla^{\mathcal {F}}=p\nabla^{T(SM)}p,\quad \nabla^{{\mathcal
{F}}^{\bot}}=p^{\bot}\nabla^{T(SM)}p^{\bot}.\label{000}
\end{align}
Let $g^{\mathcal{F}}$, $g^{{\mathcal{F}}^{\bot}}$ be the restriction of $g^{T(SM)}$ on $\mathcal {F}$, $\mathcal
{F}^{\bot}$ respectively.
Then $\nabla^{\mathcal {F}}$, $\nabla^{{\mathcal {F}}^{\bot}}$ are metric-preserving connections
of $\mathcal {F}$, $\mathcal
{F}^{\bot}$ respectively.

Now the Bott connection $\widetilde{\nabla}^{{\mathcal {F}}^{\bot}}$ on
$\mathcal{F}^{\bot}$ is determined by the following definition
\begin{definition}[{cf. [7], [11, Sect. 1.7]}]
 For any $X\in\Gamma(T(SM))$, $U\in\Gamma({\mathcal {F}}^{\bot})$,

{\upshape(i)}\quad If $X\in\Gamma({\mathcal {F}})$, set
$\widetilde{\nabla}^{{\mathcal {F}}^{\bot}}_XU=p^{\bot}[X,U]$;

{\upshape(ii)}\quad If $X\in\Gamma({\mathcal {F}}^{\bot})$, set
$\widetilde{\nabla}^{{\mathcal {F}}^{\bot}}_XU=\nabla^{{\mathcal
{F}}^{\bot}}_XU$.
\end{definition}
In general, the Bott connection $\widetilde{\nabla}^{{\mathcal
{F}}^{\bot}}$ is not a metric-preserving connection of $g^{{\mathcal
{F}}^{\bot}}$. One defines the dual connection
$\widetilde{\nabla}^{{\mathcal {F}}^{\bot},*}$ of the Bott
connection as follows,
$$d\langle U,V\rangle=\langle\widetilde{\nabla}^{{\mathcal {F}}^{\bot}}U,V\rangle+\langle U,\widetilde{\nabla}^{{\mathcal {F}}^{\bot},*}V\rangle,$$
where $U,V\in\Gamma({\mathcal {F}}^{\bot})$.

Following Bismut-Zhang [4, p.62] and Liu-Zhang \cite{LiuZ}, set
\begin{align}
2H=\widetilde{\nabla}^{{\mathcal
{F}}^{\bot},*}-\widetilde{\nabla}^{{\mathcal {F}}^{\bot}}\quad {\rm
and} \quad \widehat{\nabla}^{{\mathcal
{F}}^{\bot}}=\widetilde{\nabla}^{{\mathcal
{F}}^{\bot}}+H.\label{300}
\end{align}
It is known that the connection $\widehat{\nabla}^{{\mathcal {F}}^{\bot}}$ is the
symmetrization of the Bott connection with respect to the metric $g^{{\mathcal{F}}^{\bot}}$
on $\mathcal{F}^{\bot}$ and so a
metric-preserving connection on $\mathcal{F}^{\bot}$.

 Some basic properties of the $\Omega^1(SM)$-valued endomorphism $H$ are
 also established in [4, p.62] and \cite{LiuZ}.
\begin{lem}[{[4, p.62], \cite{LiuZ}}]\label{lemma 3}
 For any $U,V\in\Gamma({\mathcal {F}}^{\bot})$, one has that

{\upshape(1)}\quad
$\langle HU,V\rangle$=$\langle U,HV\rangle$,

{\upshape(2)}\quad $H(U)=0$,

{\upshape(3)}\quad $H={1\over 2}(g^{{\mathcal{F}}^{\bot}})^{-1}\widetilde{\nabla}^{{\mathcal
{F}}^{\bot}}g^{{\mathcal {F}}^{\bot}}$.
\end{lem}
Write $H=(H_{ij})$ under the local frame (\ref{mathbf e}).
As a corollary of Lemma \ref{lemma 3}, one has that $H_{ij}=H_{ji}$ and
$H_{ij}=H_{ij\gamma}\omega^{n+\gamma}$ for some functions $H_{ij\gamma}$.
\begin{lem}\label{lemma 4}
Set
$A_{ijk}=\frac{1}{4}F[F^2]_{y^iy^jy^k}$. With respect to (\ref{ei to delta delta x}), one has
\begin{eqnarray}
H_{ij\gamma}=-A_{pqk}u_i^pu_j^qu_{\gamma}^k.
                                              \label{H under natrual basis}
\end{eqnarray}
Moreover, $H_{ij\gamma}=0$ if $i=n$ or $j=n$.
\end{lem}
\begin{proof}
For any $X\in\Gamma({\mathcal {F}})$ and $U,V\in\Gamma(\mathcal{F}^{\bot})$,
one gets easily that
\begin{align*}
\langle 2H(X)U,V\rangle=(\mathcal{L}_Xg^{{\mathcal{F}}^{\bot}})(U,V).
\end{align*}
So by (\ref{gij}), (\ref{ei to delta delta x}) and (\ref{H under natrual basis}),
\begin{align*}
H_{ij\gamma}&=\langle H(\mathbf{e}_{n+\gamma})\mathbf{e}_{i},\mathbf{e}_{j}\rangle
={1\over 2}(\mathcal{L}_{\mathbf{e}_{n+\gamma}}g^{{\mathcal{F}}^{\bot}})(\mathbf{e}_{i},\mathbf{e}_{j})
={1\over 2}(\mathcal{L}_{\mathbf{e}_{n+\gamma}}g)(\mathbf{e}_{i},\mathbf{e}_{j})\\
&={1\over 2}({\mathbf{e}_{n+\gamma}}g_{pq})dx^p\otimes dx^q(\mathbf{e}_{i},\mathbf{e}_{j})
=-{1\over 2}Fu_{\gamma}^k\frac{\partial g_{pq}}{\partial
y^k}dx^p\otimes dx^q(\mathbf{e}_{i},\mathbf{e}_{j})\\
&=-{1\over 4}F[F^2]_{y^py^qy^k}u_i^pu_j^qu_{\gamma}^k=-A_{pqk}u_i^pu_j^qu_{\gamma}^k.
\end{align*}

By the Euler lemma, it is clear that $H_{ij\gamma}=0$ if $i=n$ or
$j=n$.
\end{proof}
\begin{remark}
Traditionally, the Cartan tensor is defined as
$\mathbf{A}=A_{ijk}dx^i\otimes dx^j\otimes dx^k$, and the Cartan
form is that $\mathbf{I}=g^{ij}A_{ijk}dx^k:=A_{k}dx^k$ (cf. [8,
p.11-12]). From this reason, we call $H$ the Cartan endomorphism,
and still call the one form $\eta={\rm tr}[H]\in\Omega^1(SM)$ the
Cartan form for a Finsler manifold $(M,F)$.
\end{remark}
Let $\bm{\omega}=(\omega_j^i)$ be the connection matrix of the Bott
connection with respect to the orthnormal frame (\ref{mathbf e}),
i.e.,
\begin{align}
\tD{\mathcal{F}^{\bot}}\mathbf{e}_i=\omega_i^j\mathbf{e}_j.\label{400}
\end{align}

\begin{theo}
The connection matrix $\bm{\omega}=(\omega_j^i)$ in (\ref{400}) of the Bott
connection is the unique solution of the following structure
equations,
\begin{equation}\left\{
\begin{aligned}
&d\vartheta=\vartheta\wedge\bm{\omega},\\
&\bm{\omega}+\bm{\omega}^t=-2H,
\end{aligned}\right.\label{Chern connection structure eq. matrix}
\end{equation}
where $\vartheta=(\omega^1,\ldots,\omega^n)$.
\end{theo}
\begin{proof}
For any $X,Y\in\Gamma(T(SM)$,
\begin{align*}
&(d\omega^i-\omega^j\wedge\omega_j^i)(X,Y)\\
=&X(\omega^i(Y))-Y(\omega^i(X))-\omega^i([X,Y])-\left(\omega^j(X)\omega_j^i(Y)-\omega^j(Y)\omega_j^i(X)\right).
\end{align*}
Now for any $X,Y\in\Gamma(\mathcal{F})$, and $U,V\in\Gamma(\mathcal{F}^{\bot})$,
one has
\begin{align*}
(d\omega^i-\omega^j\wedge\omega_j^i)(X,Y)=-\omega^i([X,Y])=0,
\end{align*}
\begin{align*}
&(d\omega^i-\omega^j\wedge\omega_j^i)(X,U)=X(\omega^i(U))+\omega^j(U)\omega_j^i(X)-\omega^i([X,U])\\
=&\omega^i\left(\widetilde{\nabla}^{{\mathcal
{F}}^{\bot}}_XU-[X,U]\right)=0,
\end{align*}
and
\begin{align*}
&(d\omega^i-\omega^j\wedge\omega_j^i)(U,V)\\
=&\left(U(\omega^i(V))+\omega^j(V)\omega_j^i(U)\right)-\left(V(\omega^i(U))+\omega^j(U)\omega_j^i(V)\right)-\omega^i([U,V])\\
=&\omega^i\left(\widetilde{\nabla}^{{\mathcal
{F}}^{\bot}}_UV-\widetilde{\nabla}^{{\mathcal
{F}}^{\bot}}_VU-[U,V]\right)=\omega^i\left(\D{T(SM)}_UV-\D{T(SM)}_VU-[U,V]\right)=0.
\end{align*}
Hence the Bott connection matrix $\bm{\omega}$ satisfies the first equation of (\ref{Chern connection
structure eq. matrix}).

The second equation of (\ref{Chern connection structure eq. matrix})
comes directly from the definition of $H$.

To prove the uniqueness, let
$\widetilde{\bm{\omega}}=(\widetilde{\omega}_j^i)$ be another
solution of (\ref{Chern connection structure eq. matrix}). One has
\begin{align*}
\omega^j\wedge(\widetilde{\omega}_j^i-\omega_j^i)=0.
\end{align*}
It deduces that
\begin{align*}
\widetilde{\omega}_j^i-\omega_j^i=a_{jk}^i\omega^k,\quad {\rm
with}\quad a_{jk}^i=a_{kj}^i.
\end{align*}
From the second equation of (\ref{Chern connection structure eq.
matrix}), one has that
\begin{align*}
0=(\omega_j^i+\omega_i^j)-(\widetilde{\omega}_j^i+\widetilde{\omega}_i^j)
=(a_{jk}^i+a_{ik}^j)\omega^k,
\end{align*}
and so $a_{jk}^i+a_{ik}^j=0$. Thus
$$(a_{jk}^i+a_{ik}^j)+(a_{ij}^k+a_{kj}^i)-(a_{ki}^j+a_{ji}^k)=2a_{jk}^i=0.$$
So we conclude that $\widetilde{\omega}_j^i-\omega_j^i=0$.
\end{proof}
\begin{mcor}
The connection forms of the Bott connection in (\ref{400}) satisfy
$$\omega_{\alpha}^n=-\omega^{\alpha}_n=\omega^{n+\alpha},\quad{\rm and}\quad \omega^n_n=0.$$
\end{mcor}
\begin{proof}
The formula $\omega^n_n=0$ comes directly from Lemma 4. By Lemma 1, the connection
forms $\omega_{\alpha}^n$ can be written as
\begin{align}
\omega_{\alpha}^n=\omega^{n+\alpha}+c_{\alpha\beta}\omega^{\beta}\quad\mbox{\rm with}
\quad c_{\alpha\beta}=c_{\beta\alpha}.\label{600}
\end{align}
The second equation of
(\ref{Chern connection structure eq. matrix}) and Lemma 4 imply that
\begin{align}
\omega^{\alpha}_n=-\omega_{\alpha}^n=-\omega^{n+\alpha}-c_{\alpha\beta}\omega^{\beta}.\label{601}
\end{align}
By (\ref{601}) and the first equation of (\ref{Chern connection structure eq.
matrix}), one has that
\begin{align*}
d\omega^{\alpha}&=\omega^\beta\wedge\omega_\beta^\alpha+\omega^n\wedge\omega_n^\alpha\\
&=\omega^\beta\wedge(\omega_\beta^\alpha+c_{\alpha\beta}\omega^n)
+\omega^n\wedge(-\omega^{n+\alpha}).
\end{align*}
Set
$\widetilde{\omega}_\beta^\alpha=\omega_\beta^\alpha+c_{\alpha\beta}\omega^n$,
$\widetilde{\omega}_{\alpha}^n=-\widetilde{\omega}^{\alpha}_n=\omega^{n+\alpha}$
and $\widetilde{\omega}^n_n=0$. Clearly, $\widetilde{\bm{\omega}}=(\widetilde{\omega}_j^i)$
satisfies the first equation of (\ref{Chern connection structure
eq. matrix}). Moreover,
$$\widetilde{\omega}_\beta^\alpha+\widetilde{\omega}_\beta^\alpha
=2c_{\alpha\beta}\omega^n-2H_{\alpha\beta\gamma}\omega^{n+\gamma}.$$
Note that by Cartan homotopy formula, one has
\begin{align*}
&\mathcal{L}_{\mathbf{e}_n}g
=\mathcal{L}_{\mathbf{e}_n}\sum_{i=1}^n\omega^i\otimes\omega^i
=\sum_{i=1}^n\left((\mathcal{L}_{\mathbf{e}_n}\omega^i)\otimes\omega^i
+\omega^i\otimes(\mathcal{L}_{\mathbf{e}_n}\omega^i)\right)\\
=&\sum_{i=1}^n\left((i_{\mathbf{e}_n}d\omega^i)\otimes\omega^i
+\omega^i\otimes(i_{\mathbf{e}_n}d\omega^i)\right)
=\sum^{n-1}_{\alpha=1}\left((i_{\mathbf{e}_n}d\omega^\alpha)\otimes\omega^\alpha
+\omega^\alpha\otimes(i_{\mathbf{e}_n}d\omega^\alpha)\right)\\
=&\sum^{n-1}_{\alpha,\beta=1}
\left(i_{\mathbf{e}_n}(\omega^\beta\wedge{\widetilde\omega^\alpha_\beta})\otimes\omega^\alpha
+\omega^\alpha\otimes
i_{\mathbf{e}_n}(\omega^\beta\wedge{\widetilde\omega^\alpha_\beta})
+i_{\mathbf{e}_n}(\omega^n\wedge{\widetilde\omega^\alpha_n})\otimes\omega^\alpha
+\omega^\alpha\otimes i_{\mathbf{e}_n}(\omega^n\wedge{\widetilde\omega^\alpha_n})\right)\\
=&-\sum^{n-1}_{\alpha,\beta=1}({\widetilde\omega}^\alpha_\beta(\mathbf{e}_n)+{\widetilde\omega}^\beta_\alpha(\mathbf{e}_n))
\omega^\alpha\otimes\omega^\beta-\sum^{n-1}_{\alpha=1}(\omega^\alpha\otimes\omega^{n+\alpha}+
\omega^{n+\alpha}\otimes\omega^\alpha).
\end{align*}
Comparing with Lemma 2, we conclude that
$$2c_{\alpha\beta}=(\widetilde{\omega}_\alpha^\beta
+\widetilde{\omega}_\beta^\alpha)(\mathbf{e}_n)=0.$$
Now by (\ref{600}), the corollary is proved.
\end{proof}
\begin{remark}
Noticed that the Chern connection is defined by the structure
equations (\ref{Chern connection structure eq. matrix}) (cf. [3,
p.38], [6, p.282], [8, p.23-33]), so the Bott connection in our case
is exactly the Chern connection. In this case, we partially answer a
question of M. Abate and G. Patrizio (cf. [1, p.29]). Moreover,
under the orthnormal frame (\ref{mathbf e}), the symmetrization
$\widehat{\nabla}^{\mathcal{F}^{\bot}}$ of the Bott connection has
the connection matrix
\begin{align}
\widehat{\bm{\omega}}=\bm{\omega}+H.\label{100}
\end{align}
In [3, p.39], an expression of the Cartan connection is given in the
local coordinate system on $SM$. One can check easily that these two
expressions are differ from a gauge transformation of the
connection. So $\widehat{\nabla}^{\mathcal{F}^{\bot}}$ turns out to
be the Cartan connection in Finsler geometry.
\end{remark}

Now we consider the rescaled metrics on $SM$ with $\epsilon>0$,
\begin{align}
g^{T(SM),\epsilon}=\frac{1}{\epsilon^2}\sum_{i=1}^n\omega^i\otimes\omega^i
+\sum_{\alpha=1}^{n-1}\omega^{n+\alpha}\otimes\omega^{n+\alpha}.\label{rescaled
metric on T(SM)}
\end{align}

Let $\nabla^{T(SM),\epsilon}$ be the Levi-Civita connection of
$g^{T(SM),\epsilon}$ and $\nabla^{{\mathcal
F}^{\bot},\epsilon}=p^{\bot}\nabla^{T(SM),\epsilon}p^{\bot}$.

Following Liu-Zhang \cite{LiuZ} and Zhang [11, Sect. 1.7], the
Cartan connection $\widehat{\nabla}^{\mathcal{F}^{\bot}}$ now can
also be obtained through the adiabatic limit technique, i.e.,
\begin{prop}
Let
$\nabla^{{\mathcal{F}}^\bot,\epsilon}=p^{\bot}\nabla^{T(SM),\epsilon}p^{\bot}$,
then
$$\lim\limits_{\epsilon\rightarrow0}\nabla^{{\mathcal {F}}^{\bot},\epsilon}=\widehat{\nabla}^{{\mathcal {F}}^{\bot}}.$$
\end{prop}

Furthermore, by using the technique of the adiabatic limit, we can
prove the following property of the Cartan endomorphism $H$.

\begin{prop}\label{H under conformal change}
Let $(M,F)$ be a Finsler manifold. For any
$\sigma\in\pi^*C^{\infty}(M)$, let ${\bar
g}^{T(SM)}=e^{2\sigma}g^{T(SM)}$ and $\bar H$ be the associated
Cartan endomorphism, then
$${\bar H}=H.$$
\end{prop}
\begin{proof}
Let $\overline{\widetilde{\nabla}^{\mathcal{F}^{\bot}}}$ and
$\overline{\widehat{\nabla}^{\mathcal{F}^{\bot}}}$ be the
Bott connection and its symmetrization corresponding ${\bar g}^{T(SM)}=e^{2\sigma}g^{T(SM)}$, respectively.
Then the corresponding Cartan endomorphism $\bar H$ is
\begin{align*}
\bar
H&=\overline{\widehat{\nabla}^{\mathcal{F}^{\bot}}}-\overline{\widetilde{\nabla}^{\mathcal{F}^{\bot}}}.
\end{align*}

Consider the rescaled coformal metrics
$${\bar g}^{T(SM),\epsilon}=\frac{1}{\epsilon^2}e^{2\sigma}g^{\mathcal{F}^{\bot}}\oplus e^{2\sigma}g^{\mathcal{F}}$$
and the projection connections $\overline{\nabla^{{\mathcal
{F}}^{\bot},\epsilon}}$ on ${\mathcal F}^{\bot,\epsilon}$.
It is clear that ${\bar H}(U)=H(U)=0$ for any $U\in\Gamma(\mathcal{F}^{\bot})$. For any
$X\in\Gamma(\mathcal{F})$, $U,V\in\Gamma(\mathcal{F}^{\bot})$, we have
\begin{eqnarray*}
\langle{\bar
H}(X)U,V\rangle&=&\langle\overline{{\widehat\nabla}_X^{{\mathcal
F}^{\bot}}}U,V\rangle
-\langle[X,U],V\rangle\\
&=&\lim_{\epsilon\rightarrow0}\langle\overline{\nabla^{\mathcal{F}^{\bot},\epsilon}_{X}}U,V\rangle
-\langle[X,U],V\rangle\\
&=&\lim_{\epsilon\rightarrow0}\frac{1}{2}e^{-2\sigma}\epsilon^2\left\{X\langle U,V\rangle_{\sigma,\epsilon}+U\langle
X,V\rangle_{\sigma,\epsilon}-V\langle X,U\rangle_{\sigma,\epsilon}\right.\\
&&\left.+\langle[X,U],V\rangle_{\sigma,\epsilon}-\langle[X,V],U\rangle_{\sigma,\epsilon}
-\langle[U,V],X\rangle_{\sigma,\epsilon}\right\}-\langle[X,U],V\rangle\\
&=&\lim_{\epsilon\rightarrow0}\frac{1}{2}\left\{X\langle
U,V\rangle+2X(\sigma)\langle U,V\rangle\right.\\
&&\left.+\langle[X,U],V\rangle-\langle[X,V],U\rangle-\epsilon^2\langle[U,V],X\rangle\right\}-\langle[X,U],V\rangle\\
&=&\frac{1}{2}\left\{X\langle
U,V\rangle-\langle[X,U],V\rangle-\langle[X,V],U\rangle\right\}\\
&=&\langle H(X)U,V\rangle.
\end{eqnarray*}
\end{proof}

\section{Geometric classes of Finsler manifolds}

Let $(M,F)$ be an oriented and closed Finsler manifold of dimension
$n$. As in the previous section, let
$\widetilde{\nabla}^{\mathcal{F}^{\bot}}$ and
 $\widehat{\nabla}^{\mathcal{F}^{\bot}}$ denote the Chern connection and
the Cartan connection on
$\mathcal{F}^{\bot}=H(SM)$, respectively.

Let
$\nabla^{\mathcal{F}^{\bot}}_t$, $t\in[0,1]$, be a family of connections on
$\mathcal{F}^{\bot}$ defined by
\begin{equation*}
\nabla^{\mathcal{F}^{\bot}}_t=(1-t)\widetilde{\nabla}^{\mathcal{F}^{\bot}}+t\widehat{\nabla}^{\mathcal{F}^{\bot}}
=\widetilde{\nabla}^{\mathcal{F}^{\bot}}+tH.
\end{equation*}
Let $R^{\mathcal{F}^{\bot}}_t=(\nabla^{\mathcal{F}^{\bot}}_t)^2$ be the curvature
of $\nabla^{\mathcal{F}^{\bot}}_t$.
The term
\begin{equation}
-n\int_0^1{\rm tr}\left[H(R^{\mathcal{F}^{\bot}}_t)^{n-1}\right]dt\label{Chern-Simons
term}
\end{equation}
appears naturally in the transgression formula associated to ${\rm
tr}\left[(R^{\mathcal{F}^{\bot}}_t)^n\right]$ (cf. [11, p.16]).

With respect to (\ref{mathbf e}), the curvature two forms of $R^{\mathcal{F}^{\bot}}_0$
are $\Omega_j^i=d\omega_j^i-\omega_j^k\wedge\omega_k^i$.
By the first equation of (\ref{Chern connection structure eq. matrix})
(also the lemma 1.14 in \cite{Zhang}), one can write
$\Omega_j^i$ as
\begin{align}
\Omega_j^i=R_{jkl}^i\omega^k\wedge\omega^l+P_{jk\gamma}^i\omega^k\wedge\omega^{n+\gamma},\label{curvatur of Chern conncetion}
\end{align}
where $R_{jkl}^i$ and $P_{jk\gamma}^i$ are some functions on $SM$.

In the following we will compute the term (\ref{Chern-Simons term})
for a Finsler surface.
\begin{theo}\label{Chern-Simons in dim 2}
Let $(M,F)$ be an oriented and closed Finsler surface. The term
(\ref{Chern-Simons term}) is given by
\begin{equation}
-2\int_0^1{\rm tr}\left[HR^{\mathcal{F}^{\bot}}_t\right]dt=\eta\wedge d\eta,
\end{equation}
where $\eta=\tr[H]=H_{111}\omega^3$ is the Cartan form of
$(M,F)$ (cf. Remark 3).
\end{theo}
\begin{proof}
Firstly one has that
\begin{eqnarray*}
\int_0^1{\rm tr}\left[HR^{\mathcal{F}^{\bot}}_t\right]dt
&=&\int_0^1{\rm tr}\left[HR^{\mathcal{F}^{\bot}}_{0}+
tH\left[\nabla^{\mathcal{F}^{\bot}}_{0},H\right]+t^2H^3\right]dt\\
&=&{\rm tr}\left[HR^{\mathcal{F}^{\bot}}_{0}+
\frac{1}{2}H\left[\nabla^{\mathcal{F}^{\bot}}_{0},H\right]+\frac{1}{3}H^3\right].
\end{eqnarray*}
In the case of $\dim M=2$, by Corollary 1 and (\ref{curvatur of Chern conncetion}), one has
$$d\omega^3=-R_{212}^1\omega^1\wedge\omega^2-P_{211}^1\omega^1\wedge\omega^{3}.$$
With respect to the local frame (\ref{mathbf e}), one gets
\begin{equation*}
H=\left[\begin{array}{cc}H_{111}\omega^3&0\\
0&0\end{array}\right],\quad
R^{\mathcal{F}^{\bot}}_{0}=\left[\begin{array}{cc}\Omega_1^1&*\\
*&*\end{array}\right]=\left[\begin{array}{cc}H_{111}R_{212}^1\omega^1\wedge\omega^2+\cdots&*\\
*&*\end{array}\right].
\end{equation*}
Thus
$$HR^{\mathcal{F}^{\bot}}_{0}\!=\!\!\left[\begin{array}{cc}\!(H_{111})^2R_{212}^1\omega^1\wedge\omega^2\wedge\omega^3&0\\
0&0\end{array}\right],\!\quad\!\!\! H[\nabla^{\mathcal{F}^{\bot}}_{0},H]\!=\!\!\left[\begin{array}{cc}\!-(H_{111})^2R_{212}^1\omega^1\wedge\omega^2\wedge\omega^3&0\\
0&0\end{array}\right],$$ and
\begin{align*}
\int_0^1{\rm
tr}\left[-HR^{\mathcal{F}^{\bot}}_t\right]dt=\frac{1}{2}(H_{111})^2R_{212}^1\omega^1\wedge\omega^2\wedge\omega^3.
\end{align*}
On the other hand,
\begin{eqnarray*}
\eta\wedge d\eta=H_{111}\omega^3\wedge
d(H_{111}\omega^3)=-(H_{111})^2R_{212}^1\omega^1\wedge\omega^2\wedge\omega^3.
\end{eqnarray*}
So Theorem 2 follows.
\end{proof}

\begin{remark} In \cite{Szabo}, Szab\'{o} proved that any two dimensional
Berwald manifold is either locally Minkowskian or Riemannian. So the
term (\ref{Chern-Simons term}) is identically zero for any two
dimensional Berwald manifold. On the other hand, in \cite{Bryant},
Bryant constructed a family of two dimensional non-Riemannian
Finsler manifolds with $R_{212}^1=1$. From Theorem 2, the cohomology class
associated to the term (\ref{Chern-Simons term}) of these Finsler
manifolds are not zero.
\end{remark}

Motivated by Theorem \ref{Chern-Simons in dim 2} and Remark 5, we
make the following definition.
\begin{definition}
For a closed and oriented Finsler manifold $(M,F)$ of dimension $n$, the top form
$\eta\wedge(d\eta)^{n-1}$
on $SM$ is called the Chern-Simons type form of $(M,F)$. The
corresponding class $$[\eta\wedge(d\eta)^{n-1}]\in H_{\rm dR}^{2n-1}(SM)$$ is
called the Chern-Simons type secondary class of $(M,F)$.
\end{definition}
When $(d\eta)^k=0$ for some $k\geq 1$, one gets a closed form
$\eta\wedge(d\eta)^{k-1}$ and so a class
$[\eta\wedge(d\eta)^{k-1}]\in H_{\rm dR}^{2k-1}(SM)$.  It would be
interesting to explore the properties of the Finsler manifolds with
$(d\eta)^k=0$ and $[\eta\wedge(d\eta)^{k-1}]\neq 0$.

Note that the form $\eta\wedge(d\eta)^{n-1}$ is unchanged about
the conformal metrics in Proposition \ref{H under
conformal change}. In the following proposition, a condition on conformal Finsler metrics
is given which leaves $\eta$ unchanged.

\begin{prop}
Let $(M,F)$ be a Finsler manifold. Let $\bar{F}=e^{\sigma}F$ be a
conformal deformation of $F$, where $\sigma\in\pi^*C^{\infty}(M)$.
Let $\eta$ and $\bar{\eta}$ be the Cartan forms of $(M,F)$ and
$(M,\bar{F})$, respectively. Then $\bar{\eta}=\eta$ if and only if
$\sigma$ satisfies
\begin{align}
\mathbf{G}(\sigma)\mathbf{I}+\mathbf{A}(\mathbf{I}^*,d\sigma^*)=0\quad{\text
and}\quad \langle \mathbf{I}^*,d\sigma^*\rangle=0,
                           \label{conditions for eta invariant under c.t.}
\end{align}
where $\mathbf{G}=\frac{y^i}{F}\frac{\delta}{\delta x^i}$ is the
Reeb vector field on $SM$; $\mathbf{A}$ is the Cartan tensor and
$\mathbf{I}$ is the usual Cartan form (cf. Remark 3);
$\mathbf{I}^*$, $d\sigma^*$ are the dual vector fields of
$\mathbf{I}$, $d\sigma$ with respect to the metric $g^{T(SM)}$,
respectively.
\end{prop}
\begin{proof}
By (\ref{geodesic coeffi.}), one has
$\overline{G}^{i}=G^i+\sigma_{x^k}y^ky^i-\frac{1}{2}F^2\sigma_{x^k}g^{ki}$.
Furthermore,
\begin{align*}
\frac{\partial \overline{G}^{i}}{\partial y^j}=\frac{\partial
G^{i}}{\partial
y^j}+\sigma_{x^j}y^i+\sigma_{x^k}y^k\delta_j^i-FF_{y^j}\sigma_{x^k}g^{ki}+FA_{pqj}g^{ip}g^{qk}\sigma_{x^k},
\end{align*}
and
\begin{align*}
\overline{\delta
y^i}&=\frac{1}{\bar{F}}\left(dy^i+\frac{\partial \overline{G}^{i}}{\partial y^j}dx^j\right)\\
&=e^{-\sigma}\frac{1}{F}\left[dy^i+\left(\frac{\partial
G^{i}}{\partial
y^j}+\sigma_{x^j}y^j+\sigma_{x^k}y^k\delta_j^i-FF_{y^j}\sigma_{x^k}g^{ki}+FA_{pqj}g^{ip}g^{qk}\sigma_{x^k}\right)dx^j\right]\\
&=e^{-\sigma}\delta
y^i+e^{-\sigma}\frac{1}{F}\left(\sigma_{x^j}y^i+\sigma_{x^k}y^k\delta_j^i-FF_{y^j}\sigma_{x^k}g^{ki}+FA_{pqj}g^{ip}g^{qk}\sigma_{x^k}\right)\delta
x^j.
\end{align*}
Corresponding to $\bar{F}$, one has that
$\overline{\omega^i}=e^{\sigma}\omega^i$ and
$\overline{\omega^{n+\gamma}}=-e^{\sigma}v_j^i\overline{\delta
y^j}$. Now,
\begin{align*}
-\overline{\omega^{n+\gamma}}&=v_j^{\gamma}\delta
y^j+\frac{1}{F}v_j^{\gamma}(\sigma_{x^k}y^j+\sigma_{x^l}y^l\delta_k^j-FF_{y^k}\sigma_{x^l}g^{lj}+FA_{pqk}g^{jp}g^{ql}\sigma_{x^l})\delta x^k\\
&=-\omega^{n+\gamma}+\frac{1}{F}\sigma_{x^l}y^lv_k^{\gamma}\delta x^k-g^{lj}v_j^{\gamma}\sigma_{x^l}F_{y^k}\delta x^k+v_j^{\gamma}A_{pqk}g^{jp}g^{ql}\sigma_{x^l}\delta x^k\\
&=-\omega^{n+\gamma}+\frac{1}{F}\sigma_{x^l}y^l\omega^{\gamma}+v_j^{\gamma}A_{pqk}g^{jp}g^{ql}\sigma_{x^l}\delta
x^k-g^{lj}v_j^{\gamma}\sigma_{x^l}\omega^n.
\end{align*}
On the other hand, one sees easily from (\ref{H under natrual basis})
that functions
$H_{\alpha\beta\gamma}$ are unchanged under the above conformal
deformations. Finally we obtain
\begin{align*}
\bar{\eta}&=\bar{H}_{ii\gamma}\overline{\omega^{n+\gamma}}\\
&=H_{ii\gamma}\omega^{n+\gamma}-\frac{1}{F}\sigma_{x^l}y^lH_{ii\gamma}\omega^{\gamma}-H_{ii\gamma}v_j^{\gamma}A_{pqk}g^{jp}g^{ql}\sigma_{x^l}\delta x^k+H_{ii\gamma}v_j^{\gamma}g^{lj}\sigma_{x^l}\omega^n\\
&=\eta+\frac{y^l}{F}\sigma_{x^l}\mathbf{I}+A_jg^{jp}A_{pqk}g^{ql}\sigma_{x^l}\delta x^k-A_jg^{jl}\sigma_{x^l}\omega^n\\
&=\eta+\mathbf{G}(\sigma)\mathbf{I}+\mathbf{A}(\mathbf{I}^*,d\sigma^*)-\langle\mathbf{I}^*,d\sigma^*\rangle\omega^n.
\end{align*}
\end{proof}


By the above proposition, the Chern-Simons type form
$\eta\wedge(d\eta)^{n-1}$ is a conformal invariant when the
conformal factor $\sigma$ satisfies (\ref{conditions for eta
invariant under c.t.}).

It should be noted that the second equation in (\ref{conditions for
eta invariant under c.t.}) also appears as the conformal invariance
condition of the so called S-curvature (cf. [2, p.231]).

\end{document}